\documentclass[preprint,12pt]{elsarticle}

\usepackage{amsfonts}
\usepackage{amsmath}
\usepackage{stmaryrd}
\usepackage{amssymb}
\usepackage{graphicx}
\usepackage{subfigure}
\usepackage{appendix}
\usepackage{url}

\newtheorem{theorem}{Theorem}[section]
\newtheorem{cor}{Corollary}[theorem]
\newtheorem{definition}{Definition}[section]

\newtheorem{lemma}{Lemma}[section]

\newtheorem{remark}{Remark}[section]

\journal{Computer Aided Geometric Design}

\begin{document}

\makeatletter
\newcommand*{\rom}[1]{\expandafter\@slowromancap\romannumeral #1@}
\makeatother

\begin{frontmatter}


\title{Isotopic Equivalence from B\'ezier Curve Subdivision}
\author[JL]{J. Li}
\address[JL]{Department of Mathematics,
       University of Connecticut, Storrs. }
       
 \author[TJP]{T. J. Peters \footnote{This author was partially supported by NSF grants CCF 0429477, CMMI 1053077 and CNS 0923158.  All statements here are the responsibility of the author, not of the National Science Foundation.}}
 \address[TJP]{Department of Computer Science and Engineering,
        University of Connecticut, Storrs.}
        
 \author[JAR]{J. A. Roulier}
 \address[JAR]{Department of Computer Science and Engineering,
        Univerisity of Connecticut, Storrs.}

\date{\today}


\begin{abstract}
We prove that the control polygon of a B\'ezier curve $\mathcal{B}$ becomes homeomorphic and ambient isotopic to $\mathcal{B}$ via subdivision, and we provide closed-form formulas to compute the number of subdivision iterations to ensure these topological characteristics. We first show that the {\em exterior angles} of control polygons converge exponentially to zero under subdivision. 
\end{abstract}

\begin{keyword}
B\'ezier curve \sep subdivision \sep piecewise linear approximation \sep non-self-intersection \sep homeomorphism \sep ambient isotopy.

\vspace{1ex}

\MSC 57Q37 \sep 57M50 \sep 57Q55 \sep 68R10
\end{keyword}

\end{frontmatter}

\section{Introduction}
\label{sec:intro}

Preserving certain topological characteristics such as homeomorphism and ambient isotopy, between an initial geometric model and its approximation, is of contemporary interest in geometric modeling \cite{Amenta2003, L.-E.Andersson2000, Lance2009, Moore_Peters_Roulier2007}, with the focus here being on Bezier curves. 

A B\'ezier curve is characterized by an indexed set of  points, which form a PL approximation of the curve, called a control polygon (Definition~\ref{def:bez}). The de Casteljau algorithm \cite{G.Farin1990} is a subdivision algorithm associated to B\'ezier curves which recursively generates control polygons more closely approximating the curve under Hausdorff distance \cite{J.Munkres1999}. We focus on homeomorphism and ambient isotopy between a B\'ezier curve and the control polygon. The control polygon homeomorphic to a simple B\'ezier curve is also simple, so homeomorphism precludes undesired self-intersections, while the control polygon ambient isotopic to a B\'ezier curve has the same knot type as the B\'ezier curve. 

However, there may be substantial topological differences between B\'ezier curves and their control polygons. First of all, B\'ezier curves and their control polygons are not necessarily homeomorphic. There are examples in the literature showing simple B\'ezier curves with self-intersecting control polygons or self-intersecting B\'ezier curves with simple control polygons \cite{JL2012, M.Neagu_E.Calcoen_B.Lacolle2000, Piegl}. Secondly, B\'ezier curves and their control polygons are not necessarily ambient isotopic. There is an example showing an unknotted B\'ezier curve with a knotted control polygon \cite{Bisceglio, Moore2006}. Examples of a knotted B\'ezier curve with an unknotted control polygon have recently appeared \cite{JL2012, Carlo}. 

Computationally, it is known that the convergence in Hausdorff distance is exponential \cite{Lane_Riesenfeld1980, Nairn-Peters-Lutterkort1999}. We show that the angular convergence rate is also exponential, and this becomes a useful tool in determining classical topological equivalence (by homeomorphism) as well as for knot equivalence. Consequently the convergence for homeomorphic and isotopic equivalence is also exponential. Furthermore, we derive closed-form formulas to compute sufficient numbers of subdivision iterations to achieve homeomorphism and ambient isotopy respectively. These formulas rely upon the constructive geometric proofs presented here. 

\section{Related Work}
\label{sec:relw}
Exponential convergence in Hausdorff distance under B\'ezier curve subdivision has been studied in the literature \cite{Lane_Riesenfeld1980, Nairn-Peters-Lutterkort1999}. Morin and Goldman proved that the discrete derivatives of the control polygons converge exponentially to the derivatives of the B\'ezier curve, by showing that discrete differentiation commutes with subdivision \cite{Morin_Goldman2001}. Our angular convergence is based on these previous results. 

The established topological equivalence by homeomorphism was given \cite{M.Neagu_E.Calcoen_B.Lacolle2000} by invoking the hodograph\footnote{The derivative of a B\'ezier curve is also expressed as a B\'ezier curve, known as the {\em hodograph} \cite{G.Farin1990}.}, but did not provide the number of subdivision iterations. We provide a constructive geometric proof for specified numbers of subdivision iterations to first produce a control polygon homeomorphic and, later, ambient isotopic to a given B\'ezier curve. Topologically reliable approximation in terms of homeomorphism of composite B\'ezier curves was established \cite{ChoMaekawa1996}, which used algorithmic techniques that do not completely rely upon the de Casteljau algorithm, but techniques related to ``significant points". As we mentioned in the introduction, topological preservation can be used to prevent undesired self-intersections. The intersection of curves and surfaces is one of the fundamental problems in areas of geometric modeling \cite{patrikalakis2002shape}. For intersections between two B\'ezier curves, C. K. Yap gives a complete subdivision algorithm \cite{yap2006complete}. 

We construct a tubular neighborhood for a B\'ezier curve, with the boundary of the tubular neighborhood being a pipe surface. Pipe surfaces have been studied since the 19th century \cite{Monge}, but the presentation here follows a contemporary source~\cite{Maekawa_Patrikalakis_Sakkalis_Yu1998}. These authors perform a thorough analysis and description of the end conditions of open spline curves.  The junction points of a B\'ezier curve are merely a special case of that analysis. 

Ambient isotopy is a stronger notion of equivalence than homeomorphism. An earlier algorithm \cite{TJP2011} establishes an isotopic approximation over a broad class of parametric geometry, but does so at the expense of the {\it a priori} bounds provided here by restricting to subdivision on splines. Other recent papers \cite{Burr2012, LineYap2011} present algorithms to compute isotopic PL approximation for $2D$ algebraic curves.  Computational techniques for establishing isotopy and homotopy have been established regarding algorithms for point-cloud  by ``distance-like functions'' \cite{Chazal2005}. 

Ambient isotopy under subdivision was previously established \cite{Moore_Peters_Roulier2007} for $3D$ B\'ezier curves of low degree (less than 4), where a crucial unknotting condition was trivially established for these low degrees. The results presented here extend to B\'ezier curves of arbitrary degree, by a more refined analysis of avoiding knots locally within the PL approximation generated. The focus on higher degree versions was motivated by applications in molecular simulation where B\'ezier curve models are created on input of  hundreds of thousands of points, with interest in having curves that are at least $C^1$.  Preserving that continuity over low degree models on this magnitude of points would be extremely tedious.

Denne and Sullivan proved that for homeomorphic curves, if their distance and angles between first derivatives are within some given bounds, then these curves are ambient isotopic \cite{DenneSullivan2008}. We use this result to derive ambient isotopy for B\'ezier curves and provide formulas to compute the number of subdivision iterations, which is computationally crucial, as B\'ezier curves are used in many practical areas. 

\section{Definitions and Notation}
\label{sec:def-no-thm}

Mathematical definitions, notation and a fundamental supportive theorem are presented in this section.  More specialized definitions will follow in appropriate sections.  The standard Euclidean norm will be denoted by $|| \hspace{2ex} ||$.
\begin{definition}
\label{def:bez}
The parameterized \textbf{B\'ezier curve}, denoted as $\mathcal{B}(t)$, of degree $n$ with control points $p_j \in\mathbb{R}^3$ is defined by
$$\mathcal{B}(t) = \sum_{j=0}^{n}{B_{j,n}(t)p_j}, t\in[0,1],
$$
where $B_{j,n}(t)=\left(\!\!\!
  \begin{array}{c}
	n \\
	j
  \end{array}
  \!\!\!\right)t^j(1-t)^{n-j}$ and the PL curve given by the points $\{p_0,p_1,\ldots,p_n\}$ is called its \textbf{control polygon}. When $p_0=p_n$, the control polygon is closed. Otherwise when $p_0 \neq p_n$, it is open. 
\end{definition}

In order to avoid technical considerations and to simplify the exposition, the class of B\'ezier curves considered will be restricted to those where the first derivative never vanishes.
\begin{definition}
\label{def:regularity}
A differentiable curve is said to be \textbf{regular} if its first derivative never vanishes.
\end{definition}

\begin{definition}
\label{def:simple}
A curve is said to be \textbf{simple} if it is non-self-intersecting.
\end{definition}

The B\'ezier curve of Definition~\ref{def:bez} is typically called a \textit{single segment B\'ezier curve}, while a \textit{composite B\'ezier curve} is created by joining two or more single segment B\'ezier curves at their common end points. 

\vspace{1ex}
\textbf{We use $\mathcal{B}$ to denote a simple, regular, $C^1$, compact, composite B\'ezier curve in $\mathbb{R}^3$, throughout the paper.}
\vspace{1ex}

 \begin{definition} \textup{\cite{J.Munkres1999} }
Let $X$ and $Y$ be two non-empty subsets of a metric space $(M, d)$. The \textbf{Hausdorff distance} $\mu(X, Y)$ is defined by
 $$\mu(X,Y):= \max\{\sup_{x\in X} \inf_{y\in Y} d(x,y), \sup_{y\in Y} \inf_{x\in X} d(x,y) \}.$$
 \end{definition}

Subdivision algorithms are fundamental for B\'ezier curves \cite{G.Farin1990} and a brief overview is given here.
Figure~\ref{fig:sub-ex} shows the first step of the de Casteljau algorithm with an input value of $\frac{1}{2}$ on a single B\'ezier curve.  For ease of exposition, the de Casteljau algorithm with this value of $\frac{1}{2}$ is assumed, but other fractional values can be used with appropriate minor modifications to the analyses presented. The initial control polygon $P$ is used as input to generate local PL approximations,  $P^1$ and $P^2$, as Figure~\ref{fig:dec1} shows.  Their union, $P^1 \cup P^2$, is then a new PL curve whose Hausdorff distance is closer to the curve than that of $P$.  

\begin{figure}[h!]
\centering

        \subfigure[Subdivision process]
   {   \includegraphics[height=2.7cm]{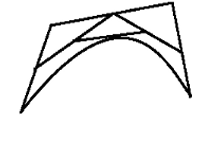}  \label{fig:dec2} }
           \subfigure[Initial and resultant curves]
   {   \includegraphics[height=2.7cm]{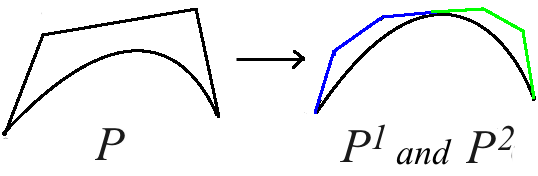}  \label{fig:dec1} }
\caption{A subdivision with parameter $\frac{1}{2}$}\label{fig:sub-ex}
 \end{figure}

A summary is that subdivision proceeds by selecting the midpoint of each segment of $P$ and  these midpoints are connected to create new segments, as Figure~\ref{fig:dec2} shows.
Recursive creation and connection of midpoints continues until a final midpoint is selected.  The union of the segments from the last step then forms a new PL curve. Termination is guaranteed since $P$ has only finitely many segments.

After $i$ iterations, the subdivision process generates $2^i$ PL sub-curves, each being a control polygon for part of the original curve \cite{G.Farin1990}, which will be referred to as a \textbf{sub-control polygon}\footnote{Note that by the subdivision process, each sub-control polygon of a simple B\'ezier curve is open.}, denoted by $P^k$ for $k = 1, 2, 3, \ldots, 2^i$.   Each $P^k$ has $n$ points and their union $\bigcup_k P^k$ forms a new PL curve that converges in Hausdorff distance to approximate the original B\'ezier curve.  The B\'ezier curve defined by $\bigcup_k P^k$ is exactly the same B\'ezier curve defined by the original control points $\{p_0,p_1,\ldots,p_n\}$ \cite{Lane_Riesenfeld1980}. So $\bigcup_k P^k$ is a new control polygon of the B\'ezier curve. 

Exterior angles were defined \cite{Milnor1950} in the context of closed PL curves, but are adapted here for both closed and open PL curves.  Exterior angles unify the concept of total curvature for curves that are PL or  differentiable.  


\begin{definition} \textup{\cite{Milnor1950}}
\label{def:exterior_angles}
The \textbf{exterior angle} between two oriented line segments, denoted as $\overrightarrow{p_{m-1}p_m}$ and $\overrightarrow{p_mp_{m+1}}$, is the angle formed by $\overrightarrow{p_mp_{m+1}}$ and the extension of $\overrightarrow{p_{m-1}p_m}$. Let the measure of the exterior angle to be $\alpha_m$ satisfying:
\begin{center}
$0 \leq \alpha_m \leq \pi$.
\end{center}

\end{definition}

\begin{definition}
Parametrize a curve $\gamma(s)$ with arc length $s$ on $[0,\ell]$. Then its \textbf{total curvature} is $\int_0^{\ell} || \gamma''(s)|| \ ds.$
\end{definition}

Total curvature can be defined for both $C^2$ and PL curves.  In both cases, the total curvature is denoted by $T_{\kappa}(\cdot)$. The unified terminology is invoked in Fenchel's theorem, which is fundamental to the work presented here.

\begin{definition}\textup{\cite{Milnor1950}}
\label{def:dcurva}
The \textbf{total curvature} of a PL curve in $\mathbb{R}^3$ is the sum of the measures of the exterior angles.
\end{definition}

Fenchel's Theorem \cite{DoCarmo1976} presented below is applicable both to PL curves and to differentiable curves.

\begin{theorem} \textup{\cite[Fenchel's Theorem]{DoCarmo1976}} \label{thm:gen-fenchel}
The total curvature of any closed curve is at least $2 \pi$, with equality holding if and only if the curve is convex.
\end{theorem}

Denote a PL curve with vertices $\{p_0,p_1,\ldots,p_n\}$ by $P$, and the uniform parametrization \cite{Morin_Goldman2001} of $P$ over $[0,1]$ by $l(P)_{[0,1]}$. That is:
$$
l(P)_{[0,1]}(\frac{j}{n})=p_j\ for\ j=0,1,\cdots,n
$$
and $l(P)_{[0,1]}$ interpolates linearly between vertices.

\begin{definition}\textbf{Discrete derivatives} \cite{Morin_Goldman2001} are first defined at the parameters $t_j=\frac{j}{n}$, where
$$l(P)_{[0,1]}(t_j)=p_j$$ for $j=0,1,\cdots,n-1$. Let
$$
p'_j=l'(P)_{[0,1]}(t_j)=\frac{p_{j+1}-p_j}{t_{j+1}-t_j}.
$$
Denote $P'=(p'_0,p'_1,\ldots,p'_{n-1}).$ Then define the discrete derivative for $l(P)_{[0,1]}$ as:
$$
l'(P)_{[0,1]}=l(P')_{[0,1]}.
$$
For simplicity of notation, we let $\mathcal{P}(t)=l(P)_{[0,1]}$ and $\mathcal{P}'(t)=l'(P)_{[0,1]}$. 
\end{definition}

\section{Angular Convergence under Subdivision} 
\label{sec:ac}
We use the notation established in Section~\ref{sec:def-no-thm}. Also, let $\mathcal{P}_0(t)$ denote the original control polygon before subdivision. Let $M$ be the maximum of the distance between two consecutive vertices of $\mathcal{P}'_0(t)$.  let $\mathcal{P}(t_{j-1})$ and $\mathcal{P}(t_j)$ be any consecutive vertices of a control polygon $\mathcal{P}$ obtained by subdivision. 


\begin{lemma}
\label{lem:derivative_diff}
For a $C^1$, composite B\'ezier curve $\mathcal{B}$, we have 
$$||\mathcal{P}'(t_j)-\mathcal{P}'(t_{j-1})|| \leq \frac{M}{2^i}.$$
\end{lemma}
\begin{proof}
Morin and Goldman \cite[Lemma 4]{Morin_Goldman2001} proved that the discrete differentiation commutes with subdivision, so $\mathcal{P}'$ can be viewed as being obtained by subdividing $\mathcal{P}'_0$. But $\mathcal{P}'_0$ is a control polygon of $\mathcal{B}'$ \mbox{\cite[Lemma 6]{Morin_Goldman2001}}. Another previous result \cite[Lemma 2.5]{Lane_Riesenfeld1980} showed that the distance between any two consecutive vertices of a control polygon is bounded by $\frac{M}{2^i}$. \hfill $\boxempty$
\end{proof}

\begin{theorem}[Angular Convergence]\label{thm:curvature_conv}
For a $C^1$, composite B\'ezier curve $\mathcal{B}$, the exterior angles of the PL curves generated by subdivision converge uniformly to $0$ at a rate of $O(\sqrt{\frac{1}{2^i}})$.
\end{theorem}

\begin{proof}
Since $\mathcal{B}(t)$ is assumed to be regular and $C^1$, the non-zero minimum of $||\mathcal{B}'(t)||$ over the compact set [0,1] is obtained. For brevity, the notations of $u_i  = \mathcal{P}'(t_j)$, $v_i = \mathcal{P}'(t_{j-1})$ and $\alpha=\alpha_m$ are introduced. The convergence of $u_i$ to $\mathcal{B}'(t_j)$ \cite{Morin_Goldman2001} implies that $||u_i||$ has a positive lower bound for $i$ sufficiently large, denoted by $\lambda$.

Lemma~\ref{lem:derivative_diff} gives that
$||u_i-v_i|| \rightarrow 0$ as $i \rightarrow \infty$ at a rate of  $O(\frac{1}{2^i})$. This implies:
$||u_i||-||v_i|| \rightarrow 0$ as $i \rightarrow \infty$ at a rate of  $O(\frac{1}{2^i})$.

Consider 
$$1-\cos(\alpha)=1-\frac{u_iv_i}{||u_i||||v_i||}$$
$$=\frac{||u_i||||v_i||-v_iv_i+v_iv_i-u_iv_i}{||u_i||||v_i||}$$
\begin{align}\label{eq:cos-1} \leq \frac{||u_i||-||v_i||}{||u_i||} + \frac{||v_i-u_i||}{||u_i||} \leq \frac{||u_i||-||v_i||}{\lambda} + \frac{||v_i-u_i||}{\lambda} \leq \frac{2||v_i-u_i||}{\lambda}\end{align}
It follows from Lemma~\ref{lem:derivative_diff} that
\begin{align}\label{eq:ml2} 1-\cos(\alpha) \leq \frac{M}{\lambda 2^{i-1}}.\end{align}

It follows from the continuity of $arc\cos$ that $\alpha$ converges to $0$ as $i \rightarrow \infty$. To obtain the convergence rate, taking the power series expansion of $\cos$ we get
\begin{alignat}{2}
1-\cos(\alpha) & \geq \alpha^2  (\frac{1}{2}-|\frac{\alpha^2}{4!}-\frac{\alpha^4}{6!}+\cdots|) \nonumber \\
&=\alpha^2  (\frac{1}{2}-\alpha^2  |\frac{1}{4!}-\frac{\alpha^2}{6!}+\cdots|)
\label{eq:cos}
\end{alignat}
Note that for $1 > \alpha$, 
\begin{align}\label{eq:e} e = 1+1+\frac{1}{2!}+\frac{1}{3!}+\frac{1}{4!}+\cdots > |\frac{1}{4!}-\frac{\alpha^2}{6!}+\cdots|.\end{align}

Combining Inequality~\ref{eq:cos} and \ref{eq:e} we have, $$1-\cos(\alpha) > \alpha^2  (\frac{1}{2}-\alpha^2  e).$$
For any $0< \tau <\frac{1}{2}$, sufficiently many subdivisions will guarantee that $\alpha$ is small enough such that
$1 > \alpha$ and $\tau > \alpha^2  e$. Thus 
$$1-\cos(\alpha)> \alpha^2(\frac{1}{2}-\alpha^2  e)> \alpha^2  (\frac{1}{2} - \tau)>0.$$
By Inequality~\ref{eq:ml2} we have
$$\alpha < \sqrt{\frac{2M}{\lambda (\frac{1}{2}-\tau)}} \sqrt{\frac{1}{2^i}}.$$
So $\alpha$ converges to $0$ at a rate of $O(\sqrt{\frac{1}{2^i}})$. \hfill $\boxempty$
\end{proof}

\section{Topologically reliable control polygons}
\label{sec:avoid-int}

We present sufficient conditions for a homeomorphism between a subdivided control polygon and its associated B\'ezier curve, and derive an ambient isotopy by relying on related results \cite{DenneSullivan2008}.

\subsection{Homeomorphism}

To obtain a homeomorphism, we first establish a local homeomorphism between a sub-control polygon and the corresponding sub-curve of $\mathcal{B}$, and then establish a global homeomorphism between the control polygon and $\mathcal{B}$. 

\begin{lemma} \textup{\cite[Lemma 7.4]{JL-isoconvthm}}
\label{lem:non-int}
Let $P$ be an open PL curve in $\mathbb{R}^3$.  If $T_{\kappa}(P)=\sum_{j=1}^{n-1}{\alpha_j}< \pi$, then $P$ is simple.
\end{lemma}

\begin{theorem}
\label{thm:non-self-int-sub}
For a $C^1$, composite B\'ezier curve $\mathcal{B}$, there exists a sufficiently large value of $i$, such that after $i$-many subdivisions, each of the sub-control polygons generated as $P^k$ for
$k = 1, 2, 3, \ldots, 2^i$ will be simple.
\end{theorem}

\begin{proof}
For each $P^k$, the measures of the exterior angles of $P^k$ converge uniformly to zero as $i$ increases (Theorem~\ref{thm:curvature_conv}).  Each open $P^k$ has $n$ edges.  Denote the $n-1$ exterior angles of each $P^k$ by $\alpha_j^k$, for $j = 1, \ldots, n-1$ and for $k = 1, 2, 3, \ldots, 2^i.$  Then  there exists $i$ sufficiently large such that
$$ \sum_{j = 1}^{n - 1} \alpha_j^k  < \pi,$$ for each $k = 1, 2, 3, \ldots, 2^i.$
Use of Lemma~\ref{lem:non-int} completes the proof.  \hfill $\boxempty$ \\
\end{proof}

The proof techniques for homeomorphism rely upon the sub-control polygons to be pairwise disjoint, except at their common end points. Denote two generated sub-control polygons of $\mathcal{B}$ as 
\begin{center}
$P=(p_0,p_1,\ldots,p_n)$ and $Q=(q_0,q_1,\ldots,q_n)$. 
\end{center}

\begin{definition}\label{def:adj}
The sub-control polygons $P$ and $Q$ are said to be \textbf{consecutive} if the last vertex $p_n$ of  $P$ is the first vertex $q_0$ of  $Q$, that is, $p_n=q_0$. 
\end{definition}

\begin{remark}
\label{re:adj-coll}
For $\mathcal{B}$, the $C^1$ assumption ensures that the segments $\overrightarrow{p_{n-1}p_n}$ and $\overrightarrow{q_0q_1}$ are collinear. The regularity assumption ensures that the exterior angle can not be $\pi$. So the exterior angle at the common point is $0$. 
\end{remark} 

Lemma~\ref{lem:plpi} extends to arbitrary degree B\'ezier curves from a previously established result that was restricted to cubic B\'ezier curves \cite{Stone_DeRose}, as used in the proof of isotopy under subdivision for low-degree B\'ezier curves \cite{Moore_Peters_Roulier2007}.

\begin{figure}[htb]
\centering
   \includegraphics[height=7.5cm]{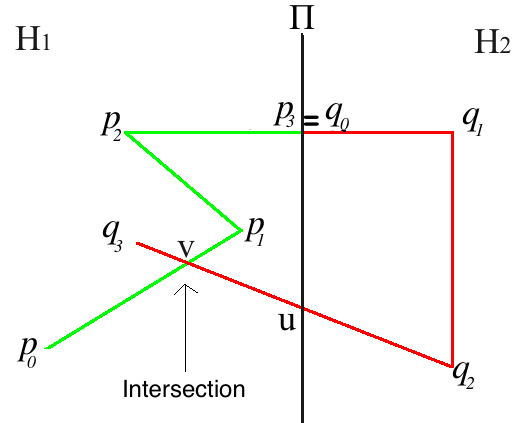}
 \vspace{-0.1in}
    \caption{Intersecting consecutive sub-control polygons}
 \label{fig:pwint}
\end{figure}

\begin{lemma}\label{lem:plpi}
Let $\Pi$ be the plane normal to a sub-control polygon at its initial vertex. If the total curvature of the sub-control polygon is less than $\frac{\pi}{2}$, then the initial vertex is the only single point where the plane intersects the sub-control polygon. 
\end{lemma}

\begin{proof}
Denote the sub-control polygon as $Q=(q_0, q_1, \ldots, q_n)$, where Figure~\ref{fig:pwint} shows an orthogonal projection of this 3D geometry. Assume to the contrary that $\Pi \cap Q$ contains a point $u$ where $u \neq q_0$. Consider the closed polygon formed by vertices $\{q_0, \ldots, u, q_0\}$. Then by Theorem~\ref{thm:gen-fenchel} we know that the total curvature of the closed polygon is at least $2\pi$. However, excluding the exterior angles at $q_0$ (which is $\frac{\pi}{2}$), and the exterior angles at $u$ (which is at most $\pi$ by Definition~\ref{def:exterior_angles}), we still at least have $\frac{\pi}{2}$ left, which contradicts to $T_{\kappa}(Q) < \frac{\pi}{2}$.\hfill $\boxempty$
\end{proof} 

\begin{lemma}
\label{lem:pi}
Recall that $\mathcal{B}$ denotes a simple, regular, $C^1$, composite B\'ezier curve in $\mathbb{R}^3$. Let $w$ be a point of $\mathcal{B}$ where $\mathcal{B}$ is subdivided and let $\Pi$ be the plane normal to $\mathcal{B}$ at $w$. Then there exists a subdivision of $\mathcal{B}$ such that the sub-control polygon ending at $w$ and the sub-control polygon beginning at $w$ intersect $\Pi$ only at the single point $w$.  
\end{lemma}

\begin{proof}
The plane $\Pi$ separates  $\mathbb{R}^3$ into two disjoint open half-spaces, denoted as $H_1$ and $H_2$, such that $\mathbb{R}^3= H_1 \hspace{1ex} \cup \hspace{1ex} \Pi \hspace{1ex} \cup \hspace{1ex} H_2$ and $H_1 \hspace{1ex} \cap \hspace{1ex} H_2 = \emptyset$. By Remark~\ref{re:adj-coll}, the exterior angle at $\{w\}$ is $0$. 

Perform sufficient many subdivisions so that the control polygon ending at $w$, denoted by $P$ , and the control polygon beginning at $w$, denoted by $Q$, each have total curvature less than $\frac{\pi}{2}$ by Theorem~\ref{thm:curvature_conv}. Therefore, by Lemma~\ref{lem:plpi} the only point where $P$ or $Q$ intersect $\Pi$ is at $w$. \hfill $\boxempty$\\
\end{proof} 

This global homeomorphism will be proven by reliance upon pipe surfaces, which are defined below. 

\begin{definition}\label{def:ps}
The \textbf{pipe surface} of radius $r$ of a parameterized curve \textbf{c}$(t)$, where $t \in [0,1]$ 
is given by
\[ \textbf{ p}(t,\theta) = \textbf{ c}(t) + r[cos(\theta) \textbf{ n}(t) + 
sin(\theta) \textbf{ b}(t) ], \]
where $\theta \in [0,2\pi]$ and $\textbf{ n}(t)$ and $\textbf{ b}(t)$ are,
respectively, the normal and bi-normal vectors at the point $\textbf{c}(t)$, 
as given by the Frenet-Serret trihedron.  The curve $\textbf{c}$ is called a spine curve.
\end{definition}

For $\mathcal{B}$ and $i$ subdivisions, with resulting sub-control polygons $P^k$ for $k=1, \ldots, 2^i$, let $S_\mathcal{B}(r)$ be a pipe surface of radius $r$ for $\mathcal{B}$  so that $S_r(\mathcal{B})$ is non-self-intersecting.  For each $k = 1, \ldots, 2^i$, denote 
\begin{itemize}
\item the parameter of the initial point of $P^k$ by $t_0^k$, and that of the terminal point by $t_n^k$
\item the normal disc of radius $r$ centered at $\mathcal{B}(t)$ as $D_r(t)$,
\item the union $\bigcup_{t\in[t_0^k,t_n^k]}D_r(t)$ by $\Gamma_k$, and designate it as a \textbf{pipe section}.
\end{itemize}

\begin{theorem}
\label{thm:non-self-intersection}
Sufficient subdivisions will yield a simple control polygon that is homeomorphic to $\mathcal{B}$.
\end{theorem}
\begin{proof}
By Theorem~\ref{thm:curvature_conv}, we can take $\iota_1$ subdivisions so that $T_{\kappa}(P^k) < \pi/2$, for each sub-control polygon $P^k$. By Lemma~\ref{lem:non-int}, this choice of $\iota_1$ guarantees that each $P^k$ is simple. By the convergence in Hausdorff distance under subdivision \cite{Nairn-Peters-Lutterkort1999}, we can take $\iota_2$ subdivisions such that  the control polygon generated by $\iota_2$ subdivision fits inside the pipe surface $S_r(\mathcal{B})$. Choose $\iota=\max\{\iota_1,\iota_2\}$. By Lemma~\ref{lem:pi}, this choice of $\iota$ ensures that each $P^k$ fits inside the corresponding $\Gamma_k$. This plus the fact that $P^k$ is simple shows that the control polygon, $\bigcup_{k = 1}^{2^i} P^k$, is simple, which implies the homeomorphism. \hfill $\boxempty$
\end{proof}

\subsection{Ambient isotopy}
We derive the ambient isotopy following \cite[Proposition 3.1]{DenneSullivan2008}.

\begin{cor}\label{corol:amb}
Sufficient subdivisions will yield a simple control polygon that is ambient isotopic to $\mathcal{B}$.
\end{cor}
\begin{proof}
By Theorem~\ref{thm:non-self-intersection}, sufficiently many subdivisions will produce a homeomorphic $\mathcal{P}$. We can define a homeomorphism $h$ mapping $\mathcal{P}(t)$ to $\mathcal{B}(t)$ by 
$$h(p)=\mathcal{B}(\mathcal{P}^{-1}(p))\ \ \text{for}\ p \in \mathcal{P}.$$

Denne and Sullivan \cite[Proposition 3.1]{DenneSullivan2008} proved that provided the homeomorphism, $\mathcal{B}$ and $\mathcal{P}$ are ambient isotopic $||\mathcal{B}(t)-\mathcal{P}(t)||<\frac{r}{2}$ (where $r$ is the radius of a pipe surface) and $\max_{t\in[0,1]} \theta(t)<\frac{\pi}{6}$ (where $\theta(t)$ is the angle between $\mathcal{B}'(t)$ and $\mathcal{P}'(t)$). 

Because $\mathcal{P}(t)$ converges to $\mathcal{B}(t)$ \cite{Nairn-Peters-Lutterkort1999} and $\mathcal{P}'(t)$ converges to $\mathcal{B}'(t)$ \cite{Morin_Goldman2001}, the conclusion follows.  \hfill $\boxempty$
\end{proof}

\begin{remark}
The result \cite[Proposition 3.1]{DenneSullivan2008} contains an assumption that the limit curve is $C^{1,1}$, to ensure the existence of a positive thickness, which is equivalent to the existence of a non-self-intersecting pipe surface here. Note that our limit curve $\mathcal{B}$ is assumed to be a simple, compact composite B\'ezier curve. So the curve is $C^2$ except at finitely many points. It follows easily that $\mathcal{B}$ is actually $C^{1,1}$. 
\end{remark}


\section{Sufficient Subdivision Iterations}
\label{sec:nneeded}

In this section, we shall establish closed-form formulas to computer sufficient numbers of subdivisions for small exterior angles, homeomorphism and ambient isotopy respectively. 

From the previous sections we know that the homeomorphism is obtained by subdivision based on two criteria: (1) angular convergence; and (2) convergence in distance. So the speed of achieving these topological characteristics is determined by the angular convergence rate and the convergence rate in distance which are both exponential. Here, we further find closed-form formulas to compute sufficient numbers of subdivision iterations to achieve these properties. 

\begin{definition}
Let $P$ denote a control polygon of a B\'ezier curve, and let $P_x$ denote an ordered list of all of $x$-coordinates of  $P$ (with similar meaning given to $P_y$ for the $y$-coordinates and to $P_z$ for the $z$-coordinates).  Let
\[\parallel \triangle_2P_x \parallel_{\infty}=\max_{0<m<n}| P_{m-1,x}-2P_{m,x}+P_{m+1,x}| \]
be the maximum absolute second difference of the x-coordinates of control points,  (with similar meanings for the $y$ and $z$ coordinates) .  Let
\[ \Delta_2 P = ( \parallel \triangle_2P_x \parallel_{\infty}, \parallel \triangle_2P_y \parallel_{\infty}, \parallel \triangle_2P_z\parallel_{\infty}), \]
(i.e.) a vector with 3 values.  
\end{definition}

\begin{definition}
The distance\footnote{The distance here is as previously used \cite{Nairn-Peters-Lutterkort1999}. Note that the distance is not smaller than Fr\'echet distance. Our following results remain true if this distance is changed to Fr\'echet distance.} \textup{\cite{Nairn-Peters-Lutterkort1999}} between a B\'ezier curve $\mathcal{B}$ and the control polygon $\mathcal{P}$ generated by $i$ subdivisions is given by
$$\max_{t\in[0,1]}||\mathcal{P}(t)-\mathcal{B}(t)||. $$
\end{definition}  

\begin{lemma}
\label{lem:dist-conv}
The distance between the B\'ezier curve and its control polygon after $i$th-round subdivision is bounded by
\begin{equation}
\label{eq:Ninfty}
\frac{1}{2^{2i}}N_{\infty}(n)|| \Delta_2 P ||,
\end{equation}
where $$N_{\infty}(n)=\frac{\lfloor n/2 \rfloor \cdot \lceil n/2 \rceil}{2n}.$$ 
\end{lemma}

\begin{proof}
A published lemma \cite [Lemma 6.2] {Nairn-Peters-Lutterkort1999} proves a similar result restricted to scalar valued
polynomials. We consider coordinate-wise and apply this result to the $x, y$, and $z$ coordinates
respectively, so that the distance of the $x$-coordinates of the B\'ezier curve and its control
polygon after ith-round subdivision is bounded by 
$$\frac{1}{2^{2i}}N_{\infty}(n)\parallel \Delta_2P_x \parallel_{\infty},$$
with similar expressions for the $y$ and $z$ coordinates.  Taking the Euclidean norm of the indicated three $x, y$ and $z$ bounds yields the upper bound given by~(\ref{eq:Ninfty}), an upper
bound of the distance between the B\'ezier curve and its control polygon after the $i$th subdivision.  \hfill $\boxempty$ \\
\end{proof}

For convenience, denote the above bound in distance as:
\begin{align}\label{eq:B_dist} B_{dist}(i) := \frac{1}{2^{2i}}N_{\infty}(n) || \Delta_2P ||.\end{align}

\begin{lemma}\label{lem:deri-dist}
After $i$ subdivision iterations, the distance between $\mathcal{P}'$ and $\mathcal{B}'$ is bounded by $B'_{dist}(i)$, where
\begin{align}\label{eq:B'_dist} B'_{dist}(i) := \frac{1}{2^{2i}}N_{\infty}(n-1) || \Delta_2P' ||,\end{align} and $P'$ that consists of $n-1$ control points is the control polygon of $\mathcal{B}'$.
\end{lemma}
\begin{proof}
 A control polygon's derivative is again a control polygon of the B\'ezier curve's derivative \cite[Lemma 6]{Morin_Goldman2001}. So by Lemma \ref{lem:dist-conv}, we have 
\begin{align}\label{eq:1st-deri} \max_{t\in[0,1]}||\mathcal{P}'(t)-\mathcal{B}'(t)|| \leq B'_{dist}(i). \end{align} \hfill $\boxempty$
\end{proof}
\subsection{Subdivision iterations for small exterior angles}
Assume $\nu$ is a small measure of angle between $0$ and $\pi$. We shall find how many subdivisions will generate a control polygon such that the measure $\alpha$ of each exterior angle satisfies\begin{align}\label{eq:alth} \alpha < \nu.  \end{align}

Recall the proof of angular convergence (Theorem~\ref{thm:curvature_conv}). Consider two arbitrary consecutive derivatives $u_i  = \mathcal{P}'(t_j)$ and $v_i = \mathcal{P}'(t_{j-1})$ and the corresponding exterior angle $\alpha$. Recall that in Section~\ref{sec:ac} we had Inequalities~\ref{eq:cos-1} and~\ref{eq:ml2}:
\begin{align}\label{eq:cosb}1-\cos(\alpha)\leq \frac{2||v_i-u_i||}{||u_i||} \leq \frac{M}{||u_i|| 2^{i-1}}.\end{align}

Let $\sigma=\min\{||\mathcal{B}'(t)||: t \in [0,1]\}$. The regularity of $\mathcal{B}$ ensures that $\sigma > 0$ and the continuity of $\mathcal{B}'$ on the compact interval $[0,1]$ ensures that the minimum exists. Recall  $u_i=\mathcal{P}'(t_j)$ for some $t_j\in [0,1]$. So it follows from Inequality~\ref{eq:1st-deri} that
$$||\mathcal{B}'(t_j)|| - ||u_i|| \leq B'_{dist}(i).$$
Solving the inequality we get 
$$||u_i|| \geq ||\mathcal{B}'(t_j)||- B'_{dist}(i) \geq \sigma-B'_{dist}(i).$$
In order to have $u_i \neq 0$, it is sufficient to perform enough subdivisions such that $$||u_i|| \geq  \sigma -B'_{dist}(i) >0,$$ that is $B'_{dist}(i) <\sigma$. By the definition (Equation~\ref{eq:B'_dist}) of $B'_{dist}(i)$ we set, 
$$\frac{1}{2^{2i}}N_{\infty}(n-1)\parallel \triangle_2P' \parallel < \sigma.$$
Therefore for $B'_{dist}(i) <\sigma$, it suffices to have\footnote{Throughout this paper, we use $\log$ for $\log_2$.} \begin{align}\label{eq:N1} i > \frac{1}{2} \log(\frac{N_{\infty}(n-1)\parallel \triangle_2P' \parallel}{\sigma})=N_1. \end{align}

After the $i$ subdivision iterations, whenever $i > N_1$, then $B'_{dist}(i)<B'_{dist}(N_1)$, because $B'_{dist}(i)$ is a strictly decreasing function (Equation~\ref{eq:B'_dist}). So it follows from Inequality~\ref{eq:cosb} that whenever $i > N_1$,
$$
1-\cos(\alpha)\leq \frac{M}{2^{i-1}(\sigma-B'_{dist}(i))}.
$$

To obtain $\alpha < \nu$ (Inequality~\ref{eq:alth}), it suffices to have that  $1 - cos(\alpha) < 1 - cos(\nu)$. Now choose $i$ large enough so that
\begin{align}\label{eq:bdgc} 1-\cos(\alpha)\leq \frac{M}{2^{i-1}(\sigma-B'_{dist}(N_1))} < 1-\cos(\nu).\end{align}
The second inequality of Inequality~\ref{eq:bdgc} implies that
$$i > \log(\frac{2M}{(1-\cos(\nu))(\sigma-B'_{dist}(N_1))}).$$
To simplify this expression, let 
\begin{align}\label{eq:f} f(\nu)=\frac{2M}{(1-\cos(\nu))(\sigma-B'_{dist}(N_1))}. \end{align}
Then, we have $$i > \log(f(\nu)).$$

\begin{theorem}
\label{thm:ntheta}
Given any $\nu > 0$, there exists an integer $N(\nu)$ defined by 
\begin{align}\label{eq:N} N(\nu)=\lceil \max\{N_1,\log(f(\nu))\} \rceil \end{align}
where $N_1$, and $f(\nu)$ are given by Equations~\ref{eq:N1} and~\ref{eq:f} respectively, 
such that each exterior angle is less than $\nu$, whenever $i> N(\nu)$.  
\end{theorem}

\begin{proof}
It follows from the definitions of $N_1$ and $f(\nu)$ and the analysis in this section.  \hfill $\boxempty$
\end{proof}

It is worth to note that $N$ is a logarithm depending on several parameters such as $\sigma$, $N_{\infty}(n)$ and $\triangle_2P'$ as well as an upper bound variable $\nu$. 

\subsection{Subdivision iterations for homeomorphism}
For a regular B\'ezier curve $\mathcal{B}$ of degree $1$ or $2$, the control polygon is trivially\footnote{For degree 1, both the curve and the polygon are either a  point or a line segment. For degree 2, there are three points. The curve and the polygon are planar and open (otherwise the curve is not regular). } ambient isotopic to  $\mathcal{B}$. We consider $n \geq 3$.

Given any $\nu > 0$, Theorem~\ref{thm:ntheta} shows that there exists an integer $N(\nu)$, such that each exterior angle is less than $\nu$ after $N(\nu)$ subdivisions.  Furthermore, there is an explicit closed formula to compute $N(\nu)$.
\begin{theorem}\label{thm:sloc}
There exists a positive integer, $N(\frac{\pi}{n-1})$ for $n > 2$, where $N(\frac{\pi}{n-1})$ is defined by Equation~\ref{eq:N}, such that after $\lceil N(\frac{\pi}{n-1}) \rceil$ subdivisions, each sub-control polygon will be simple.
\end{theorem}
\begin{proof} 
By Theorem~\ref{thm:ntheta}, after $N(\frac{\pi}{n-1})$ subdivisions,  each exterior angle is less than $\frac{\pi}{n-1}$. Since each sub-control polygon has a $n-1$ exterior angles, the total curvature of each sub-control polygon is less than $\pi$. Lemma~\ref{lem:non-int} implies that this is a sufficient condition for each sub-control polygon being simple.  
\hfill $\boxempty$\\
\end{proof}

While existence of sufficiently many iterations for the control polygon to fit inside the pipe $S_r(\mathcal{B})$ has been established, it remains of interest to bound the number of subdivisions that are sufficient for this containment. Define $N'(r)$ by 
\begin{align}\label{eq:N'}N'(r)=\frac{1}{2}\log(\frac{N_{\infty}(n)|| \Delta_2P ||}{r}), \end{align}
where $r$ is the radius of a non-self-intersecting pipe surface for $\mathcal{B}$.
By the definition of $B_{dist}(i)$ (Equation~\ref{eq:B_dist}) and Equation~\ref{eq:N'}, we have $B_{dist}(i) < r$ whenever $i>N'(r)$.

\begin{lemma}\label{lem:nofdist}
The control polygon generated by $i$ subdivisions, where $i > N'(r)$ and $N'(r)$ is given by Equation~\ref{eq:N'}, satisfies 
$$\max_{t\in[0,1]}||\mathcal{B}(t)-\mathcal{P}(t)|| < r,$$
and hence fits inside the pipe surface of radius $r$ for $\mathcal{B}$. 
\end{lemma}

\begin{proof}
By Lemma~\ref{lem:dist-conv}, $\max_{t\in[0,1]}||\mathcal{B}(t)-\mathcal{P}(t)|| \leq B_{dist}(i)$. Then this lemma follows from the definition of $N'(r)$ given by Equation~\ref{eq:N'}.\hfill $\boxempty$
\end{proof}

While Theorem~\ref{thm:sloc} addresses each sub-control polygon, it is of interest to ensure that the union of all these sub-control polygons is also simple.  In Theorem~\ref{thm:simples}, that union is the `control polygon', as the result of multiple subdivisions.

\begin{theorem}\label{thm:simples}
Set $$\hat{N}=\max\{ N(\frac{\pi}{2(n-1)}),N'(r) \},$$ where $N(\nu)$ is defined by Equations~\ref{eq:N} and $N'(r)$ is given by Equation~\ref{eq:N'}. After $\lceil \hat{N} \rceil$ or more subdivisions, the control polygon will be homeomorphic. 
\end{theorem}

\begin{proof}
The inequality $N \geq N'(r)$ implies that the control polygon generated after the $N$th subdivision lies inside the pipe. The inequality $N \geq N(\frac{\pi}{2(n-1)})$ ensures that the total curvature of its each sub-control polygon is less than $\frac{\pi}{2}$. These two conditions are sufficient conditions for the control polygon being simple (The proof of Theorem~\ref{thm:non-self-intersection}). \hfill $\boxempty$
\end{proof}

\subsection{Subdivision iterations for ambient isotopy}
Recall, by Corollary~\ref{corol:amb}, that a homeomorphic $\mathcal{P}$ will further be ambient isotopic if \mbox{$||\mathcal{B}(t)-\mathcal{P}(t)||<\frac{r}{2}$} and $\max_{t\in[0,1]} \theta(t)<\frac{\pi}{6}$ (where $\theta(t)$ is the angle between $\mathcal{B}'(t)$ and $\mathcal{P}'(t)$). We may produce $N'(\frac{r}{2})$ subdivisions to satisfy the first condition (Lemma~\ref{lem:nofdist}). To guarantee the second condition, we consider:

$$1- \cos (\theta(t))  = 1- \frac{\mathcal{B'}(t) \cdot \mathcal{P'}(t)}{||\mathcal{B'}(t)|| \cdot ||\mathcal{P'}(t)||}$$
$$=\frac{||\mathcal{B'}(t)|| \cdot ||\mathcal{P'}(t)||- \mathcal{P'}(t) \cdot \mathcal{P'}(t)+\mathcal{P'}(t) \cdot \mathcal{P'}(t)-\mathcal{B'}(t) \cdot \mathcal{P'}(t)}{||\mathcal{B'}(t)|| \cdot ||\mathcal{P'}(t)||} $$
$$\leq \frac{||\mathcal{B'}(t)||-||\mathcal{P'}(t)||}{||\mathcal{B'}(t)||} + \frac{||\mathcal{B'}(t) - \mathcal{P'}(t)||}{||\mathcal{B'}(t)||} \leq \frac{2||\mathcal{B'}(t)-\mathcal{P'}(t)||}{\sigma},$$
where $\sigma=\min\{||\mathcal{B}'(t)||: t \in [0,1]\}$ (Recall $\sigma>0$.) From Inequality~\ref{eq:1st-deri}
$$\max_{t\in[0,1]} ||\mathcal{B}'(t) -\mathcal{P}'(t)|| \leq B'_{dist}(i),$$ we have 
$$1- \cos (\theta(t))  \leq \frac{2B'_{dist}(i)}{\sigma}.$$
To have $\theta(t)<\frac{\pi}{6}$, it suffices to set
$$\frac{2B'_{dist}(i)}{\sigma} < 1-\cos(\frac{\pi}{6})=1-\frac{\sqrt{3}}{2}.$$
By Equality~\ref{eq:B'_dist}, $$B'_{dist}(i) := \frac{1}{2^{2i}}N_{\infty}(n-1) || \Delta_2P' ||,$$ 
we get
\begin{align}\label{eq:N2} i \geq \frac{1}{2} \log (\frac{2 N_{\infty}(n-1) ||\Delta_2 P'||}{(1-\frac{\sqrt{3}}{2}) \sigma}) =N_2. \end{align}
So $N_2$ subdivision iterations will guarantee the second condition. 

\begin{theorem}
Set $$N^{\ast}=\max\{N(\frac{\pi}{2(n-1)}) , N'(\frac{r}{2}), N_2\},$$
where $N, N', N_2$ are given by Equations~\ref{eq:N}, \ref{eq:N'} and~\ref{eq:N2} respectively. 
After $\lceil N^{\ast} \rceil$ or more subdivisions,  the control polygon $\mathcal{P}$ will be ambient isotopic to the B\'ezier curve $\mathcal{B}$.
\end{theorem}

\begin{proof}
The values $N(\frac{\pi}{2(n-1)})$ and $N'(\frac{r}{2}$) are used to obtain a homeomorphism, by Theorem
~\ref{thm:simples}. And then $N_2$ is used to further obtain an ambient isotopy. \hfill $\boxempty$
\end{proof}

\begin{remark}
Note that $N(\frac{\pi}{2(n-1)})= \max\{N_1, \log{f(\frac{\pi}{2(n-1)})}\}$. Comparing $N_1$ (Equation~\ref{eq:N1}) and $N_2$, we find that $N_2 < N_1+2$. By Equation~\ref{eq:N'} we also have $N'(\frac{r}{2})< N'(r)+1$. So $N^{\ast}<\hat{N} +2$, where $\hat{N}=\max\{N(\frac{\pi}{2(n-1)}) , N'(r) \}$ is a sufficient number of subdivisions to guarantee homeomorphism. So after a homeomorphism based on Theorem~\ref{thm:simples} is attained, no more than $2$ additional subdivision iterations will be used to produce the ambient isotopy\footnote{The dissertation work \cite{JiLi} of the first author adopted an alternative, more explicit way to construct the ambient isotopy, with iteration bound of $\max\{N(\frac{\pi}{2n}) , N'(r) \}$. It was shown \mbox{\cite[Remark 4.2.7]{JiLi}} that $N(\frac{\pi}{2n})< N(\frac{\pi}{2(n-1)})+1$, so that no more than $1$ additional subdivision iteration would be used to produce the ambient isotopy after a homeomorphism is attained from Theorem~\ref{thm:simples}. However, the method here has advantages because of its direct use of subdivision versus specialized techniques. }. 
\end{remark}

\section{Conclusion}
We first proved the {\em exterior angles} of control polygons under subdivision converge to $0$ exponentially. We then showed that sufficiently many subdivisions produce a control polygon homeomorphic to the B\'ezier curve and further derived the ambient isotopy by relying upon a previous isotopy result by Denne and Sullivan. We established closed-form formulas to compute \textit{a priori} sufficient number of subdivisions to achieve these topological characteristics. These results are being applied in computer graphics, computer animation and scientific visualization, especially in visualizing molecular simulations. 

\bibliographystyle{plain}
\bibliography{ji-tjp-biblio}

\begin{thebibliography}{10}

\bibitem{Amenta2003}
N.~Amenta, T.~J. Peters, and A.~C. Russell.
\newblock Computational topology: Ambient isotopic approximation of
  2-manifolds.
\newblock {\em Theoretical Computer Science}, 305:3--15, 2003.

\bibitem{L.-E.Andersson2000}
L.~E. Andersson, S.~M. Dorney, T.~J. Peters, and N.~F. Stewart.
\newblock Polyhedral perturbations that preserve topological form.
\newblock {\em CAGD}, 12(8):785--799, 2000.

\bibitem{Bisceglio}
J.~Bisceglio, T.~J. Peters, J.~A. Roulier, and C.~H. Sequin.
\newblock Unknots with highly knotted control polygons.
\newblock {\em CAGD}, 28(3):212--214, 2011.

\bibitem{Burr2012}
M.~Burr, S.~W. Choi, B.~Galehouse, and C.~K. Yap.
\newblock Complete subdivision algorithms, {II}: Isotopic meshing of singular
  algebraic curves.
\newblock {\em Journal of Symbolic Computation}, 47:131--152, 2012.

\bibitem{Chazal2005}
F.~Chazal and D.~Cohen-Steiner.
\newblock A condition for isotopic approximation.
\newblock {\em Graphical Models}, 67(5):390--404, 2005.

\bibitem{ChoMaekawa1996}
W.~Cho, T.~Maekawa, and N.~M. Patrikalakis.
\newblock Topologically reliable approximation in terms of homeomorphism of
  composite {B}\'ezier curves.
\newblock {\em Computer Aided Geometric Design}, 13:497--520, 1996.

\bibitem{DenneSullivan2008}
E.~Denne and J.~M. Sullivan.
\newblock Convergence and isotopy type for graphs of finite total curvature.
\newblock In A.~I. Bobenko, J.~M. Sullivan, P.~Schr{\"o}der, and G.~M. Ziegler,
  editors, {\em Discrete Differential Geometry}, pages 163--174. Birkh{\"a}user
  Basel, 2008.

\bibitem{DoCarmo1976}
M.~P. do~Carmo.
\newblock {\em Differential Geometry of Curves and Surfaces}.
\newblock Prentice Hall, Upper Saddle River, NJ, 1976.

\bibitem{G.Farin1990}
G.~Farin.
\newblock {\em Curves and Surfaces for Computer Aided Geometric Design}.
\newblock Academic Press, San Diego, CA, 1990.

\bibitem{ho2001surface}
C.~Ho and E.~Cohen.
\newblock Surface self-intersection.
\newblock In {\em Mathematical methods for curves and surfaces}, pages
  183--194. Vanderbilt University, 2001.

\bibitem{TJP2011}
K.~E. Jordan, L.~E. Miller, T.~J. Peters, and A.~C. Russell.
\newblock Geometric topology and visualizing 1-manifolds.
\newblock In V.~Pascucci, X.~Tricoche, H.~Hagen, and J.~Tierny, editors, {\em
  Topological Methods in Data Analysis and Visualization}, pages 1 -- 13.
  Springer NY, 2011.

\bibitem{Lane_Riesenfeld1980}
J.~M. Lane and R.~F. Riesenfeld.
\newblock A theoretical development for the computer generation and display of
  piecewise polynomial surfaces.
\newblock {\em IEEE}, PAMI-2 NO.1, January 1980.

\bibitem{JiLi}
J.~Li.
\newblock {\em Topological and Isotopic Equivalence with Applications to
  Visualization}.
\newblock PhD thesis, University of Connecticut, U.S., 2013.

\bibitem{JL-isoconvthm}
J.~Li and T.~J. Peters.
\newblock Isotopic convergence theorem.
\newblock {\em Journal of Knot Theory and Its Ramifications}, 22(3), 2013.

\bibitem{JL2012}
J.~Li, T.~J. Peters, D.~Marsh, and K.~E. Jordan.
\newblock Computational topology counterexamples with 3{D} visualization of
  {B}'ezier curves.
\newblock {\em Applied General Topology}, 2012.

\bibitem{LineYap2011}
L.~Lin and C.~Yap.
\newblock Adaptive isotopic approximation of nonsingular curves: the
  parameterizability and nonlocal isotopy approach.
\newblock {\em Discrete \& Computational Geometry}, 45 (4):760--795, 2011.

\bibitem{Maekawa_Patrikalakis_Sakkalis_Yu1998}
T.~Maekawa, N.~M. Patrikalakis, T.~Sakkalis, and G.~Yu.
\newblock Analysis and applications of pipe surfaces.
\newblock {\em CAGD}, 15(5):437--458, 1998.

\bibitem{Lance2009}
L.~E. Miller.
\newblock {\em Discrepancy and Isotopy for Manifold Approximations}.
\newblock PhD thesis, University of Connecticut, U.S., 2009.

\bibitem{Milnor1950}
J.~W. Milnor.
\newblock On the total curvature of knots.
\newblock {\em Annals of Mathematics}, 52:248--257, 1950.

\bibitem{Monge}
G.~Monge.
\newblock {\em Application de l{'}analyse \`a la g\'eom\'etrie}.
\newblock Bachelier, Paris, 1850.

\bibitem{Moore2006}
E.~L.~F. Moore.
\newblock {\em Computational Topology of Spline Curves for Geometric and
  Molecular Approximations}.
\newblock PhD thesis, University of Connecticut, U.S., 2006.

\bibitem{Moore_Peters_Roulier2007}
E.~L.~F. Moore, T.~J. Peters, and J.~A. Roulier.
\newblock Preserving computational topology by subdivision of quadratic and
  cubic {B}\'ezier curves.
\newblock {\em Computing}, 79(2-4):317--323, 2007.

\bibitem{Morin_Goldman2001}
G.~Morin and R.~Goldman.
\newblock On the smooth convergence of subdivision and degree elevation for
  {B}\'ezier curves.
\newblock {\em CAGD}, 18:657--666, 2001.

\bibitem{J.Munkres1999}
J.~Munkres.
\newblock {\em Topology}.
\newblock Prentice Hall, 2nd edition, 1999.

\bibitem{Nairn-Peters-Lutterkort1999}
D.~Nairn, J.~Peters, and D.~Lutterkort.
\newblock Sharp, quantitative bounds on the distance between a polynomial piece
  and its {B}\'ezier control polygon.
\newblock {\em CAGD}, 16:613--631, 1999.

\bibitem{M.Neagu_E.Calcoen_B.Lacolle2000}
M.~Neagu, E.~Calcoen, and B.~Lacolle.
\newblock {B}\'ezier curves: Topological convergence of the control polygon.
\newblock {\em 6th Int. Conf. on Mathematical Methods for Curves and Surfaces,
  Vanderbilt}, pages 347--354, 2000.

\bibitem{patrikalakis2002shape}
N.~M. Patrikalakis and T.~Maekawa.
\newblock {\em Shape interrogation for computer aided design and
  manufacturing}.
\newblock Springer, 2002.

\bibitem{Piegl}
L.~Piegl and W.~Tiller.
\newblock {\em The NURBS Book}.
\newblock Springer, New York, 2nd edition, 1997.

\bibitem{Carlo}
C.~H. Sequin.
\newblock Spline knots and their control polygons with differing knottedness.
\newblock
  \url{http://www.eecs.berkeley.edu/Pubs/TechRpts/2009/EECS-2009-152.html}.

\bibitem{Stone_DeRose}
M.~Stone and T.~D. DeRose.
\newblock A geometric characterization of parametric cubic curves.
\newblock {\em ACM Transactions on Graphics}, 8(3):147--163, 1989.

\bibitem{yap2006complete}
C.~K. Yap.
\newblock Complete subdivision algorithms, i: Intersection of bezier curves.
\newblock In {\em Proceedings of the twenty-second annual symposium on
  Computational geometry}, pages 217--226. ACM, 2006.

\end{thebibliography}

\end{document}